\documentclass[12pt]{article}

\usepackage[utf8]{inputenc}
\usepackage[T1]{fontenc}

\usepackage{amsmath}
\usepackage{amssymb}
\usepackage{amsfonts}
\usepackage{mathtools}
\usepackage{bbold}

\usepackage{graphicx}
\graphicspath{{images/}}

\usepackage{booktabs}
\usepackage{float}
\usepackage{tikz}
\usetikzlibrary{arrows.meta}

\usepackage[ruled,longend]{algorithm2e}
\usepackage{pseudocode}

\usepackage{setspace}
\usepackage{parskip}
\usepackage{enumitem}
\usepackage{microtype}
\usepackage{bbm}

\usepackage{xcolor}
\usepackage[linkbordercolor=white]{hyperref}
\usepackage{url}
\usepackage{doi}

\usepackage{arxiv}

\title{Zero-Sum Two-Player Differential Game under Three Regimes}

\author{
    {\includegraphics[scale=0.06]{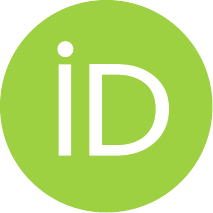}\hspace{1mm}Brahim EL ASRI}\\
    Équipe Aide à la Décision\\
    ENSA -- Université Ibn Zohr\\
    BP 1136, Morocco\\
    \texttt{b.elasri@uiz.ac.ma}
    \And
    {\includegraphics[scale=0.06]{orcid.pdf}\hspace{1mm}Magnoudéwa PAKA}\\
    Équipe Aide à la Décision\\
    ENSA -- Université Ibn Zohr\\
    BP 1136, Morocco\\
    \texttt{magnoudewa.paka@edu.uiz.ac.ma}
}

\renewcommand{\headeright}{}
\renewcommand{\undertitle}{}
\renewcommand{\shorttitle}{}

\hypersetup{
    pdftitle={Zero-Sum Two-Player Differential Game under Three Regimes},
    pdfauthor={Brahim EL ASRI, Magnoudéwa PAKA},
    pdfsubject={Differential Games},
    pdfkeywords={
        differential game,
        optimal switching,
        viscosity solutions,
        quasi-variational inequalities
    }
}

\begin{document}

\hypersetup{pdfborder=0 0 0}
\maketitle
\newtheorem{theo}{Theorem}[section]
\newtheorem{problem}{Problem}[section]
\newtheorem{pro}{Proposition}[section]
\newtheorem{cor}{Corollary}[section]
\newtheorem{axiom}{Definition}[section]
\newtheorem{rem}{Remark}[section]
\newtheorem{lem}{Lemma}[section]
\newtheorem{ex}{Example}[section]
\newtheorem{proof}{Proof}[section]
\newtheorem{Heuristics}{Heuristics}
\newtheorem{ass}{Assumption}[section]

\newcommand{\bass}{\begin{ass}}
\newcommand{\eass}{\end{ass}}
\newcommand{\bpf}{\begin{proof}}
\newcommand{\epf}{\end{proof}}
\newcommand{\brm}{\begin{rem}}
\newcommand{\erm}{\end{rem}}
\newcommand{\bethe}{\begin{theo}}
\newcommand{\eethe}{\end{theo}}
\newcommand{\bl}{\begin{lem}}
\newcommand{\el}{\end{lem}}
\newcommand{\bp}{\begin{pro}}
\newcommand{\ep}{\end{pro}}
\newcommand{\bcor}{\begin{cor}}
\newcommand{\ecor}{\end{cor}}
\newcommand{\be}{\begin{equation}}
\newcommand{\ee}{\end{equation}}
\newcommand{\beq}{\begin{eqnarray*}}
\newcommand{\eeq}{\end{eqnarray*}}
\newcommand{\beqa}{\begin{eqnarray}}
\newcommand{\eeqa}{\end{eqnarray}}
\newcommand{\dg}{\displaystyle \delta}
\newcommand{\cm}{\cal M}
\newcommand{\cF}{{\cal F}}
\newcommand{\cR}{{\cal R}}
\newcommand{\bF}{{\bf F}}
\newcommand{\tg}{\displaystyle \theta}
\newcommand{\w}{\displaystyle \omega}
\newcommand{\W}{\displaystyle \Omega}
\newcommand{\vp}{\displaystyle \varphi}
\newcommand{\ig}[2]{\displaystyle \int_{#1}^{#2}}
\newcommand{\integ}[2]{\displaystyle \int_{#1}^{#2}}
\newcommand{\produit}[2]{\displaystyle \prod_{#1}^{#2}}
\newcommand{\somme}[2]{\displaystyle \sum_{#1}^{#2}}
\newlength{\inter}
\setlength{\inter}{\baselineskip}
\setlength\parindent{24pt}
\setlength{\baselineskip}{7mm}
\newcommand{\no}{\noindent}
\newcommand{\rw}{\rightarrow}
\def \ind{1\!\!1}
\def \R{I\!\!R}
\def \N{I\!\!N}
\def \cadlag {{c\`adl\`ag}~}
\def \esssup {\mbox{ess sup}}

\hypersetup{pdfborder=0 0 0}

\maketitle


\onehalfspacing


\begin{abstract}
This paper investigates a two-player zero-sum stochastic differential game characterized by three distinct regimes, reflecting the regime-switching dynamics within the system. We aim to derive an explicit solution for a number of configurations of the switching system by means of the viscosity solutions approach. In particular, we analyze the associated Hamilton--Jacobi--Bellman--Isaacs equations and examine how the interactions between the different regimes affect the structure of the value function. The proposed framework provides a systematic approach to characterize the optimal strategies of the two players and to obtain explicit representations of the value function under suitable assumptions. We also illustrate how to derive the value functions in case we know the qualitative structure of switching regions.
\end{abstract}

\keywords{Differential game \and optimal switching \and quasi-variational inequalities \and value function \and viscosity solutions}

\noindent\textbf{AMS Classification:}
60G40; 62P20; 91B99; 91B28; 35B37; 49L25.


\section{Introduction}

In recent decades, stochastic differential games have attracted considerable attention due to their numerous applications in economics, finance, engineering, and operational research. Among these problems, zero-sum two-player switching games constitute an important class of stochastic control problems in which competing players influence the evolution of a system by switching between different operational regimes. Such models combine the strategic interactions of differential games with the flexibility of regime-switching control, providing a natural framework for studying decision-making under uncertainty when switching actions incur costs.

A zero-sum switching game describes a competitive situation in which the gain of one player is exactly the loss of the other. In the stochastic setting, the players' decisions affect both the current payoff and the future evolution of the system, leading to highly coupled optimization problems. The strategic complexity increases significantly when several operating regimes are available, since each player must determine not only the optimal timing of a switch but also the most advantageous regime to select while anticipating the opponent's actions.

The theory of regime-switching models originated with the pioneering work of Howard \textcolor{blue}{\cite{RAH}} and was further developed by Hamilton \cite{HAM} for the analysis of economic and financial systems exhibiting structural changes. Since then, regime-switching models have become a standard tool for describing systems whose dynamics evolve according to a finite collection of states. When combined with stochastic control and game-theoretic considerations, they give rise to switching games in which players modify the regime of the system through costly interventions.

Switching control and switching games have been extensively studied in the literature. Important contributions include the works of Bensoussan and Lions \textcolor{blue}{\cite{BENL}}, Øksendal and Sulem \textcolor{blue}{\cite{OS}}, Pham \textcolor{blue}{\cite{PVT}}, and Bayraktar and Yao \textcolor{blue}{\cite{BayYao}}, among others. In general, the value functions of these games are characterized as viscosity solutions of systems of variational or quasi-variational inequalities, whose analysis relies on sophisticated techniques from stochastic control theory. Existence and uniqueness results have been established in a broad setting, and several qualitative properties of the associated Nash equilibria have been investigated.

Despite these advances, obtaining explicit solutions remains a difficult task. Most existing works focus on the characterization of the value function through variational inequalities, while closed-form analytical solutions are available only in a limited number of special cases. The difficulty increases rapidly as the number of regimes grows because the number of possible switching decisions and the interactions between the players become considerably more intricate.

In \textcolor{blue}{\cite{MP-BA}}, we obtained an explicit solution for a zero-sum two-player switching game in the case of two regimes, providing an analytical characterization of the value functions and the corresponding optimal switching strategies. The present paper extends that analysis to the considerably more challenging case of three regimes. The transition from two to three regimes substantially enriches the strategic structure of the game by introducing additional switching possibilities and more complex interactions between the players.

The principal novelty of this paper is the derivation of an explicit solution for the three-regime zero-sum two-player switching game. To the best of our knowledge, explicit analytical solutions for this class of problems are scarce in the literature due to the significant mathematical complexity generated by the increased number of regimes. Our results show that moving beyond the two-regime setting gives rise to a richer strategic structure, in which the interaction between multiple switching opportunities fundamentally shapes the equilibrium strategies. The three-regime framework therefore provides new structural insights into the geometry of the switching regions and the resulting equilibrium policies, while substantially broadening the scope of the explicit analysis.

From an application point of view, the three-regime model is relevant in numerous contexts. In finance, it can describe markets evolving through bearish, neutral, and bullish phases. In economics, it models strategic decisions under different macroeconomic environments, while in engineering it applies to systems operating under several modes with different performance characteristics. The additional regime allows for a richer representation of realistic systems and more flexible modeling of strategic interactions.

The remainder of the paper is organized as follows. Section~2 introduces the mathematical framework and formulates the zero-sum switching game together with the standing assumptions. Section~3 derives the associated quasi-variational inequalities and characterizes the switching regions. Section~4 presents the explicit solution of the three-regime switching game and establishes the corresponding equilibrium strategies, in section 5 we give a numerical procedure to compute the value function
in case we know the qualitative structure of the switching regions and in section 6 we illustrate our results by numerical simulations.

\section{Problem formulation and Assumptions}

\setcounter{equation}{0}
\renewcommand{\theequation}{2.\arabic{equation}}

\subsection{Problem formulation}

The problem consists of a two-player game. The first player (Player I) and the second player (Player II) both make decisions based on the value of a state process $X := (X_t)_{t \geq 0}$.

Each player operates under several modes or regimes. Let
$\mathcal{D} := \{1,2,\dots,d\}$ denote the finite set of regimes, where each regime is identified by an index in $\{1,\dots,d\}$.

The switching costs are known from both players, and at any given time, each player knows the current regime of his counterpart. 

The regime configuration of both players is modeled by a two-dimensional càdlàg process $\textit{I}_t$ taking values in $\mathcal{D} \times \mathcal{D}$. For each value of $\textit{I}_t$, there is an associated running profit function $f_{\textit{I}_t}$, assumed to be non-negative.

The state process $X$ represents the price of a commodity and thus takes values in $\mathbb{R}_+^*$.

We denote by $J$ the payoff of the game. Player I aims to maximize $J$ by choosing a switching strategy $\xi := (\xi_n)_{n \geq 1}$ at stopping times $\tau := (\tau_n)_{n \geq 1}$, which form an increasing sequence of stopping times. Player II aims to minimize $J$ by choosing a switching strategy $\eta := (\eta_m)_{m \geq 1}$ at stopping times $\rho := (\rho_m)_{m \geq 1}$.

Switching from regime $i$ to regime $j$ incurs a cost $c_{ij}$ for Player I and $\chi_{ij}$ for Player II.

The payoff of the game is given by
\begin{equation}
\label{eq0}
J(x,\xi,\eta)
=
\mathbb{E}\Bigg[
\int_{0}^{\infty} e^{-rs}
f\big(X^{\mu_s,\nu_s}_{s}, \mu_s, \nu_s\big)\,ds
- \sum_{n \geq 1} e^{-r\tau_n} C(\xi_{n-1},\xi_n)
+ \sum_{m \geq 1} e^{-r\rho_m} \chi(\eta_{m-1},\eta_m)
\Bigg],
\end{equation}
where $r > 0$ is the discount rate.

The regime processes of the players are defined by
\[
\mu_t = \sum_{n \geq 1} \xi_n \mathbf{1}_{\{\tau_n \leq t < \tau_{n+1}\}},
\qquad
\nu_t = \sum_{m \geq 1} \eta_m \mathbf{1}_{\{\rho_m \leq t < \rho_{m+1}\}}.
\]

Hence, by definition, we have $\textit{I}_t = (\mu_t,\nu_t)$.

\subsubsection*{State dynamics}

In this work, the process $X_t$ follows a regime-switching geometric Brownian motion to ensure positivity, and is given by
\begin{equation}
\label{azz}
dX_t
=
b(X_t,\mu_t,\nu_t)\,dt
+
\sigma(X_t,\mu_t,\nu_t)\,dW_t,
\qquad
X_0 = x,
\qquad
\textit{I}_{0^-} = (i,j),
\end{equation}
where
\[
b(X_t,\mu_t,\nu_t) = b_{\textit{I}_t} X_t,
\qquad
\sigma(X_t,\mu_t,\nu_t) = \sigma_{\textit{I}_t} X_t.
\]

Throughout the paper, we refer to $b_{\textit{I}_t}$ and $\sigma_{\textit{I}_t}$ as diffusion coefficients. We also denote by $X^{x}_{I_0}$ (or simply $X^x$) the solution of the stochastic differential equation.

\subsection{Assumptions}

Throughout the paper, $d$ is a positive integer and $\mathcal{D} = \{1,\dots,d\}$.

\textbf{[H1]} The functions
$b:\mathbb{R}^+ \times \mathcal{D} \times \mathcal{D} \to \mathbb{R}^+$ and
$\sigma:\mathbb{R}^+ \times \mathcal{D} \times \mathcal{D} \to \mathbb{R}^+$
are continuous and satisfy the Lipschitz condition: there exists $C > 0$ such that for all $x,x' \in \mathbb{R}^+$ and $i,j \in \mathcal{D}$,
\begin{equation}
|\sigma(x,i,j) - \sigma(x',i,j)|
+
|b(x,i,j) - b(x',i,j)|
\leq
C|x - x'|.
\label{eqs}
\end{equation}
Moreover, they satisfy the linear growth condition:
\[
|\sigma(x,i,j)| + |b(x,i,j)| \leq C(1+|x|).
\]

\textbf{[H2]} The function
$f:\mathbb{R}^+ \times \mathcal{D} \times \mathcal{D} \to \mathbb{R}^+$
is continuous and satisfies: there exists $C > 0$ such that for all $x,x' \in \mathbb{R}^+$ and $i,j \in \mathcal{D}$,
\begin{equation}
|f(x,i,j)| \leq C(1+|x|),
\qquad
|f(x,i,j) - f(x',i,j)| \leq C|x-x'|.
\end{equation}

\textbf{[H3]} For all $i,j \in \mathcal{D}$, the switching costs $c_{ij}$ and $\chi_{ij}$ are constants satisfying the triangular condition:
\begin{equation}
c_{ik} < c_{ij} + c_{jk},
\qquad
\chi_{ik} < \chi_{ij} + \chi_{jk},
\qquad j \neq i,k.
\label{co1}
\end{equation}
This ensures that direct switching is cheaper than indirect switching through an intermediate regime.

Note that switching costs may be negative, but condition \eqref{co1} prevents arbitrage by repeated switching. In particular,
\begin{equation}
0 < c_{ij} + c_{ji},
\qquad
0 < \chi_{ij} + \chi_{ji}.
\label{co3}
\end{equation}

We may alternatively write $c(i,j)$ instead of $c_{ij}$, and $f(x,i,j)$ instead of $f_{ij}(x)$.

\textbf{[H4]} The discount rate $r$ satisfies
\[
r > \max_{i,j \in \mathcal{D}} b_{ij}.
\]
This assumption is mainly technical and ensures integrability properties needed for the explicit solution.

\textbf{[H5]} The sequences of cumulative switching costs
\[
c_n = \sum_{m=1}^{n} e^{-r\tau_m} c(\xi_{m-1},\xi_m),
\qquad
\chi_n = \sum_{l=1}^{n} e^{-r\rho_l} \chi(\eta_{l-1},\eta_l),
\]
converge almost surely, and their limits satisfy
\begin{equation}
\lim_{n \to \infty} c_n \in L^1,
\qquad
\lim_{n \to \infty} \chi_n \in L^1.
\label{L1}
\end{equation}
\subsection{Preliminary Results}

\subsubsection{Existence of solutions for the SDE}

Under assumption \textbf{[H1]}, the stochastic differential equation \eqref{azz} admits a unique strong solution.

\subsubsection{Viscosity solutions}

Let $F$ be a continuous function such that
\begin{equation}
\label{eq}
F(x,v,D_x v, D^2_{xx} v) = 0,
\end{equation}
where $D_x v$ and $D^2_{xx} v$ denote respectively the first and second derivatives with respect to $x$.

We assume that $F$ is nonincreasing with respect to its last argument, and that $x$ belongs to an open subset $\mathcal{O} \subset \mathbb{R}^+$.

\begin{axiom}[Viscosity solution]

\textbf{(i) Supersolution.}
A function $v$ is a viscosity supersolution if, for any $\bar{x} \in \mathcal{O}$ and any test function $\phi \in C^2$ such that $v - \phi$ attains a local minimum at $\bar{x}$, one has
\[
F(\bar{x}, v(\bar{x}), D_x \phi(\bar{x}), D^2_{xx} \phi(\bar{x})) \geq 0.
\]

\textbf{(ii) Subsolution.}
A function $v$ is a viscosity subsolution if, for any $\bar{x} \in \mathcal{O}$ and any $\phi \in C^2$ such that $v - \phi$ attains a local maximum at $\bar{x}$, one has
\[
F(\bar{x}, v(\bar{x}), D_x \phi(\bar{x}), D^2_{xx} \phi(\bar{x})) \leq 0.
\]

\textbf{(iii) Solution.}
A function is a viscosity solution if it is both a subsolution and a supersolution.

Equivalently, let $J^{2,+}v(x)$ and $J^{2,-}v(x)$ denote respectively the superjets and subjets of $v$ at $x$, defined by pairs $(q,X)$ such that
\[
v(y) \leq v(x) + \langle q, y-x \rangle + \frac{1}{2}\langle X(y-x), y-x \rangle + o(|y-x|^2),
\]
(resp.
\[
v(y) \geq v(x) + \langle q, y-x \rangle + \frac{1}{2}\langle X(y-x), y-x \rangle + o(|y-x|^2)).
\]

Then:
\[
v \text{ is a supersolution } \Longleftrightarrow
F(x,v,q,X) \geq 0 \quad \forall (q,X) \in J^{2,-}v(x),
\]
\[
v \text{ is a subsolution } \Longleftrightarrow
F(x,v,q,X) \leq 0 \quad \forall (q,X) \in J^{2,+}v(x).
\]

\end{axiom}

\subsubsection{Existence of the value function}

Let us define an admissible strategy in the two-player game.

\begin{axiom}[Admissible strategy]

Let $(\Omega,\mathcal{F},\mathbb{P})$ be a probability space supporting a standard $d$-dimensional Brownian motion $W=(W_t)_{t \geq 0}$, with natural filtration
\[
\mathcal{F}_t^0 := \sigma(W_s, s \leq t),
\quad
\mathbb{F} = (\mathcal{F}_t)_{t \geq 0}
\]
being the completed filtration.

We consider two players, I and II, acting on the system through switching strategies.

An admissible strategy for Player I is a sequence
\[
\delta = (\tau_m, \xi_m)_{m \geq 0},
\]
and for Player II:
\[
\nu = (\rho_n, \eta_n)_{n \geq 0},
\]
where:

\textbf{(i)} $(\tau_m)$ and $(\rho_n)$ are increasing sequences of stopping times such that $\tau_m \to \infty$ and $\rho_n \to \infty$,

\textbf{(ii)} $(\xi_m)$ and $(\eta_n)$ are $\mathcal{D}$-valued random variables such that
$\xi_m$ is $\mathcal{F}_{\tau_m}$-measurable and $\eta_n$ is $\mathcal{F}_{\rho_n}$-measurable.

\end{axiom}

We now define the upper and lower value functions.

\begin{axiom}[Upper and lower value functions]

Let $\mathcal{A}^i$ (resp. $\mathcal{B}^j$) denote the set of admissible strategies of Player I (resp. Player II).

A nonanticipative strategy for Player I is a mapping
\[
\vartheta : \bigcup_{j \in \mathcal{D}} \mathcal{B}^j \to \mathcal{A}^i,
\]
such that if two strategies $b,b'$ coincide up to a stopping time $\tau$, then $\vartheta(b)$ and $\vartheta(b')$ also coincide up to $\tau$.

Similarly, a nonanticipative strategy for Player II is a mapping
\[
\varrho : \bigcup_{i \in \mathcal{D}} \mathcal{A}^i \to \mathcal{B}^j
\]
satisfying the analogous property.

Let $\Gamma^i$ and $\Delta^j$ denote the sets of nonanticipative strategies.

The upper value is
\[
\overline{V}_{ij}
=
\inf_{\varrho \in \Delta^j}
\sup_{\varphi \in \mathcal{A}^i}
J(x,\varphi,\varrho(\varphi)),
\]
and the lower value is
\[
\underline{V}_{ij}
=
\sup_{\varphi \in \Gamma^i}
\inf_{\beta \in \mathcal{B}^j}
J(x,\varphi(\beta),\beta).
\]

\end{axiom}

\bpf
We refer to \cite{MP-BA} for further properties of the value functions.

\epf
\section{System of quasi-variational inequalities and switching regions}

We now state the Hamilton–Jacobi–Bellman–Isaacs (HJBI) system of quasi-variational inequalities.

Using the dynamic programming principle, the value function satisfies:
\begin{equation}
\label{qv}
\max \Big\{
\min \big[
r v_{ij}(x) - \mathcal{L}_{ij} v_{ij}(x) - f_{ij}(x),
\; v_{ij}(x) - M_{ij}[v](x)
\big],
\;
v_{ij}(x) - N_{ij}[v](x)
\Big\}
= 0.
\end{equation}

\begin{equation}
\label{qv2}
\min \Big\{
\max \big[
r v_{ij}(x) - \mathcal{L}_{ij} v_{ij}(x) - f_{ij}(x),
\; v_{ij}(x) - N_{ij}[v](x)
\big],
\;
v_{ij}(x) - M_{ij}[v](x)
\Big\}
= 0.
\end{equation}

where
\[
M_{ij}[v](x) = \max_{k \neq i} \{ v_{kj}(x) - c_{ik} \},
\qquad
N_{ij}[v](x) = \min_{l \neq j} \{ v_{il}(x) + \chi_{jl} \},
\]
and
\[
\mathcal{L}_{ij} v_{ij}(x)
=
\frac{1}{2} \mathrm{Tr}\big[\sigma_{ij}\sigma_{ij}^* D^2 v_{ij}(x)\big]
+
\langle b_{ij}, D v_{ij}(x) \rangle.
\]


The viscosity sub/supersolution definitions follow from standard arguments and are omitted here for brevity.

\bp

\noindent The upper and lower value functions coincide, and the value function of the stochastic differential game is given by
\[
V_{ij}(x) = \underline{V}_{ij}(x) = \overline{V}_{ij}(x),
\quad
\forall i,j \in \mathcal{D},\; x \in \mathbb{R}^+.
\]
Consequently, \eqref{qv} and \eqref{qv2} admit a unique continuous solution of polynomial growth.

\ep

\begin{proof}
We refer the reader to \textcolor{blue}{\cite{BM}}.\hfill $\Box$
\end{proof}


\noindent We now introduce the switching and continuation regions.

For Player I:
\[
S^\xi_i. =
\{ x > 0 : V_i.(x) = \max_{k \in \mathcal{D} \setminus \{i\}} (V_k.(x) - c_{ik}) \},
\]
\[
C^\xi_i. =
\{ x > 0 : V_i. (x) > \max_{k \in \mathcal{D} \setminus \{i\}} (V_k.(x) - c_{ik}) \}.
\]

For Player II:
\[
S^\eta_j. =
\{ x > 0 : V_j.(x) = \min_{l \in \mathcal{D} \setminus \{j\}} (V_l.(x) + \chi_{jl}) \},
\]
\[
C^\eta_j. =
\{ x > 0 : V_j. (x) < \min_{l \in \mathcal{D} \setminus \{j\}} (V_l.(x) + \chi_{jl}) \}.
\]

\brm
The smooth-fit property (see \cite{P}) still holds in this framework: the value function is $C^1$ on the boundaries of switching regions and $C^2$ inside continuation regions.
\erm

We finally define
\[
\mathcal{H}_{ij}[v](x) = r v(x) - \mathcal{L}_{ij} v(x) - f_{ij}(x).
\]

\bl
\label{lem:lemma3}
Let $C, S^{\xi}, S^{\eta} \subset \mathbb{R}^*_+$ be such that
\[
C = \mathbb{R}^*_+ \setminus (S^{\xi}\cup S^{\eta}),
\]
where $S^{\xi}\cup S^{\eta}$ is a union of closed subsets of $\mathbb{R}^*_+$. Fix $i,j\in\mathcal{D}$.

Let $\mathcal{Z}, h^{\xi}, h^{\eta}$ be three continuous functions on $\mathbb{R}^*_+$ such that:
\begin{itemize}
\item $\mathcal{Z}\in C^2(C)$, and $h^{\xi} \le \mathcal{Z} \le h^{\eta}$ on $C$, where $\mathcal{Z}$ solves
\begin{equation}
\label{lmeq}
\mathcal{H}^{ij}[\mathcal{Z}](x)=0.
\end{equation}

\item $\mathcal{Z}=h^{\xi}$ on $S^{\xi}$, and $\mathcal{Z}\le h^{\eta}$ on $S^{\eta}$. Moreover, $\mathcal{Z}\in C^1$ on $\partial S^{\xi}$, and $\mathcal{Z}$ is a viscosity supersolution of
\[
\mathcal{H}^{ij}[\mathcal{Z}](x)\ge 0 \quad \text{in } \operatorname{int}(S^{\xi}).
\]

\item $\mathcal{Z}=h^{\eta}$ on $S^{\eta}$, and $\mathcal{Z}\in C^1$ on $\partial S^{\eta}$. Moreover, $\mathcal{Z}$ is a viscosity subsolution of
\[
\min\big\{\mathcal{H}^{ij}[\mathcal{Z}](x),\, \mathcal{Z}(x)-h^{\xi}(x)\big\}\le 0
\quad \text{in } \operatorname{int}(S^{\eta}).
\]
\end{itemize}

Then $\mathcal{Z}$ is a viscosity solution of
\begin{equation}
\label{vs}
\max \big\{ \min\big[\mathcal{H}^{ij}[\mathcal{Z}](x),\, \mathcal{Z}(x)-h^{\xi}(x)\big],\,
\mathcal{Z}(x)-h^{\eta}(x) \big\}=0
\quad \text{on } \mathbb{R}^*_+.
\end{equation}

\el

\bpf

Let $\bar{x}\in \mathbb{R}^*_+$ and consider the following cases.

\textbf{(a) Case $\bar{x}\in C$.}

Since $\mathcal{Z}\in C^2(C)$ and solves \eqref{lmeq}, and since
\[
h^{\xi}\le \mathcal{Z}\le h^{\eta}\quad \text{on } C,
\]
it follows that $\mathcal{Z}$ satisfies \eqref{vs} classically on $C$, hence also in the viscosity sense.

\textbf{(b) Case $\bar{x}\in S^{\xi}$.}

We focus on the boundary case $\bar{x}\in \partial S^{\xi}$. We distinguish two situations.

\textbf{(i) Disjoint boundaries $\partial S^{\xi}\cap \partial S^{\eta}=\emptyset$.}

Without loss of generality, assume that $\bar{x}$ is a left boundary point of $S^{\xi}$. Then there exists $\epsilon>0$ such that
\[
(\bar{x}-\epsilon,\bar{x})\subset C,
\]
where $\mathcal{Z}\in C^2$.

Let $\phi\in C^2$ be such that $\bar{x}$ is a local minimum of $\mathcal{Z}-\phi$. Since $\mathcal{Z}\in C^1$ on $S^{\xi}$, we have
\[
\mathcal{Z}'(\bar{x})=\phi'(\bar{x}).
\]

By Taylor’s formula,
\[
\mathcal{Z}(\bar{x}-\lambda)
= \mathcal{Z}(\bar{x}) - \lambda \int_0^1 \mathcal{Z}'(\bar{x}-t\lambda)\,dt,
\]
\[
\phi(\bar{x}-\lambda)
= \phi(\bar{x}) - \lambda \int_0^1 \phi'(\bar{x}-t\lambda)\,dt.
\]

Since $\mathcal{Z}-\phi$ has a minimum at $\bar{x}$, we obtain
\begin{equation}
\label{lambdaeq}
\int_0^1 \big(\phi'(\bar{x}-t\lambda)-\mathcal{Z}'(\bar{x}-t\lambda)\big)\,dt \ge 0,
\quad \forall \lambda\in(0,\epsilon).
\end{equation}

Using \eqref{lmeq}, we have $\mathcal{H}^{ij}[\mathcal{Z}](x)=0$ on $(\bar{x}-\epsilon,\bar{x})$.
Passing to the limit yields existence of $\mathcal{Z}''(\bar{x}^-)$ and
\begin{equation}
\label{eqpf}
r\mathcal{Z}(\bar{x})
- b_{ij}\bar{x}\mathcal{Z}'(\bar{x})
- \tfrac12\sigma_{ij}^2 \bar{x}^2 \mathcal{Z}''(\bar{x}^-)
- f_{ij}(\bar{x})=0.
\end{equation}

Letting $\lambda\to 0$ in \eqref{lambdaeq}, we obtain
\[
\phi''(\bar{x}) \le \mathcal{Z}''(\bar{x}^-).
\]
Substituting into \eqref{eqpf} yields
\[
\mathcal{H}^{ij}[\mathcal{Z}](\bar{x}) \ge 0,
\]
which proves the supersolution property.

\textbf{(ii) Non-disjoint boundaries $\partial S^{\xi}\cap \partial S^{\eta}\neq \emptyset$.}

Consider $\bar{x}\in \partial S^{\xi}\cap \partial S^{\eta}$. On this set,
\[
\mathcal{Z}=h^{\xi}=h^{\eta}.
\]

If $\mathcal{H}^{ij}[\mathcal{Z}](\bar{x})\ge 0$, then
\[
\max\{\min[\mathcal{H}^{ij}[\mathcal{Z}], \mathcal{Z}-h^{\xi}], \mathcal{Z}-h^{\eta}\}=0.
\]

If $\mathcal{H}^{ij}[\mathcal{Z}](\bar{x})\le 0$, then
\[
\max\{\min[\mathcal{H}^{ij}[\mathcal{Z}], \mathcal{Z}-h^{\xi}], \mathcal{Z}-h^{\eta}\}=0.
\]

Thus, $\mathcal{Z}$ satisfies \eqref{vs} on the intersection.

\textbf{(c) Case $\bar{x}\in S^{\eta}$.}

The proof follows by symmetric arguments as in case (b).\\
\no
Therefore, $\mathcal{Z}$ is a viscosity solution of \eqref{vs}.
\hfill $\Box$
\epf

\section{Explicit Solution of the Switching Problem}
In this part, we solve the switching game problem in the three-regime case ($d=3$). We will work with nonnegative switching costs. The solutions are the functions $v_{ij}(x)$ satisfying the following system of quasi-variational inequalities:
\begin{flalign}
\max &\Big\{ \min\big[\mathcal{H}[v_{00}](x),\, v_{00}(x)-M_{00}[v](x)\big],\, v_{00}(x)-N_{00}[v](x)\Big\}=0, \\
\max &\Big\{ \min\big[\mathcal{H}[v_{01}](x),\, v_{01}(x)-M_{01}[v](x)\big],\, v_{01}(x)-N_{01}[v](x)\Big\}=0, \\
\max &\Big\{ \min\big[\mathcal{H}[v_{02}](x),\, v_{02}(x)-M_{02}[v](x)\big],\, v_{02}(x)-N_{02}[v](x)\Big\}=0, \\
\max &\Big\{ \min\big[\mathcal{H}[v_{10}](x),\, v_{10}(x)-M_{10}[v](x)\big],\, v_{10}(x)-N_{10}[v](x)\Big\}=0, \\
\max &\Big\{ \min\big[\mathcal{H}[v_{11}](x),\, v_{11}(x)-M_{11}[v](x)\big],\, v_{11}(x)-N_{11}[v](x)\Big\}=0, \\
\max &\Big\{ \min\big[\mathcal{H}[v_{12}](x),\, v_{12}(x)-M_{12}[v](x)\big],\, v_{12}(x)-N_{12}[v](x)\Big\}=0, \\
\max &\Big\{ \min\big[\mathcal{H}[v_{20}](x),\, v_{20}(x)-M_{20}[v](x)\big],\, v_{20}(x)-N_{20}[v](x)\Big\}=0, \\
\max &\Big\{ \min\big[\mathcal{H}[v_{21}](x),\, v_{21}(x)-M_{21}[v](x)\big],\, v_{21}(x)-N_{21}[v](x)\Big\}=0, \\
\max &\Big\{ \min\big[\mathcal{H}[v_{22}](x),\, v_{22}(x)-M_{22}[v](x)\big],\, v_{22}(x)-N_{22}[v](x)\Big\}=0.
\end{flalign}

We choose the following structure for the running payoff:
\begin{equation}
f_{ij}(x)=x^{\gamma}, \qquad 0<\gamma<1.
\end{equation}

As shown in \cite{MP-BA}, the function
\[
\hat{V}_{ij}(x)=\mathbb{E}\left[\int_0^{\infty} e^{-rt} f_{ij}(X_t^{x,ij})\,dt\right],
\]
is a particular solution of
\begin{equation}
\mathcal{H}[v_{ij}](x)=0.
\label{ode}
\end{equation}

In particular, when $f_{ij}(x)=x^{\gamma}$, we obtain
\[
\hat{V}_{ij}(x)=K_{ij}x^{\gamma},
\qquad
K_{ij}= \frac{1}{r - b_{ij}\gamma + \frac{1}{2}\sigma_{ij}^2\gamma(1-\gamma)} > 0,
\quad i,j\in\{0,1,2\}.
\]

Under the above assumptions, the general solution of \eqref{ode} is given by
\[
v(x)=A_{ij}x^{m_{ij}^+}+B_{ij}x^{m_{ij}^-},
\]
where $A_{ij},B_{ij}$ are constants and
\begin{equation}
m_{ij}^+
= -\frac{b_{ij}}{\sigma_{ij}^2}+\frac{1}{2}
+ \sqrt{\left(-\frac{b_{ij}}{\sigma_{ij}^2}+\frac{1}{2}\right)^2+\frac{2r}{\sigma_{ij}^2}}
>1,
\label{mu}
\end{equation}
\begin{equation}
m_{ij}^-
= -\frac{b_{ij}}{\sigma_{ij}^2}+\frac{1}{2}
- \sqrt{\left(-\frac{b_{ij}}{\sigma_{ij}^2}+\frac{1}{2}\right)^2+\frac{2r}{\sigma_{ij}^2}}
<0.
\label{md}
\end{equation}
We also assume that \[K_{ab}=K_{cd} \iff \{m_{ab}^-=m_{cd}^-\} \cap \{m_{ab}^+ = m_{cd}^+ \}. \]
With regard to \textcolor{blue}{\ref{mu}}, \textcolor{blue}{\ref{md}}, and linear growth property of the value functions (see \cite{MP-BA}), we guess that if both players are in their continuation region in the neighborhood of $0$, then $B_{ij}=0$, and if they are in their continuation region in the neighborhood of $+ \infty$, then $A_{ij}=0.$\\

\begin{theo}
\textbf{Case where $K_{ij}$ is constant}

Let $v_{ij}(x)=\hat{V}_{ij}(x)$ for $i,j\in\{0,1,2\}$. Then the $v_{ij}$ are the solutions of the system of quasi-variational inequalities with:
\[ S^\xi_{ij} =\emptyset, \qquad C^\xi_{ij} = (0;+\infty), \quad S^\eta_{ij} =\emptyset, \qquad C^\eta_{ij} = (0;+\infty). \]
\end{theo}

\begin{proof}
Since $\hat{V}_{ij}(x)$ is a particular solution of \eqref{ode}, it follows that
\[
\mathcal{H}[v_{ij}](x)=0.
\]

Moreover,
\[
v_{ij}(x)-M[v_{ij}](x)=\hat{V}_{ij}(x)-(\hat{V}_{ij}(x)-c_{ij})=c_{ij},
\]
and since $c_{ij}>0$, we have $v_{ij}(x)-M[v_{ij}](x)\geq 0$.

Similarly,
\[
v_{ij}(x)-N[v_{ij}](x)=\hat{V}_{ij}(x)-(\hat{V}_{ii}(x)+\chi_{ji})=-\chi_{ji},
\]
and since $\chi_{ji}>0$, we obtain $v_{ij}(x)-N[v_{ij}](x)\leq 0$.

\noindent Therefore, all the conditions of Lemma \ref{lem:lemma3} are satisfied, which implies that $v_{ij}(x)$ solves the QVI (system \eqref{qv}).
\hfill $\Box$
\end{proof}

\bethe
Case where $K_{0j}=Cst_1 < K_{ij}=Cst_2$, $\; i=1,2,\; j=0,1,2$.
\\
\\
Let $v_{ij}(x)$ be defined as follows:
\begin{flalign}
    v_{0j}(x)&= (\hat{V}_{0j}(x)+A_{0j}x^{m_{0j}^+})\mathbf{1}_{x<x_{0j}} + (\hat{V}_{1j}(x) - \min \{c_{01}, c_{02}\})\mathbf{1}_{x\geq x_{0j}}, \; j=0,1,2\\
    v_{ij}(x)&= \hat{V}_{ij}(x), \quad i=1,2,\; j=0,1,2.
\end{flalign}

where $x_{ij}$ and $A_{ij}$ are obtained using $\mathrm{C}^1$ property of the value functions and are defined:
\begin{equation*}
(K_{1j}-K_{0j})(x_{0j})^{\gamma}
= \frac{m_{0j}^+}{m_{0j}^+ - \gamma}\min\{c_{01},c_{02}\},
\quad
A_{0j} = (K_{1j}-K_{0j})\frac{\gamma}{m_{0j}^+}(x_{0j})^{\gamma-m_{0j}^+}.
\end{equation*}

\eethe

Such $v_{ij}(x)$ are the solutions associated to the QVI.

\[ S^\xi_{0j} =[x_{0j}, +\infty),\quad C^\xi_{0j} = (0;x_{0j}), \quad  S^\xi_{0j} =\emptyset ,\quad C^\xi_{0j} = (0; +\infty), \quad j=0,1,2.\]

\[ S^\xi_{ij} =\emptyset, \quad C^\xi_{ij} = (0;+\infty), \quad S^\eta_{ij} =\emptyset, \quad C^\eta_{ij} = (0;+\infty). \quad i=1,2, \quad j=0,1,2. \]

We can see that by design \[ v_{00}(x)=v_{01}(x)=v_{02}(x). \]
\bpf

We only verify the QVI for $v_{00}$ and $v_{10}$, since the remaining cases follow by exactly the same arguments.

\noindent
\textbf{Step 1. Verification for $v_{00}$.}

\noindent
$\bullet$ For $x<x_{00}$, we have $\mathcal{H}[v_{00}](x)=0$ by definition. On the other hand,

\[
v_{00}(x)-M_{00}[v](x)
= v_{00}(x)-\max (v_{10}(x)-c_{01},\, v_{20}(x)-c_{02}).
\]

\[
\max (v_{10}(x)-c_{01},\, v_{20}(x)-c_{02})
= v_{10}(x)-\min \{c_{01}, c_{02}\},
\]
then
\[
v_{00}(x)-M_{00}[v](x)
= \hat{V}_{00}(x)+A_{00}x^{m_{00}^+}-\hat{V}_{10}(x)+\min \{c_{01}, c_{02}\}.
\]

Define
\[
h_1(x)=\hat{V}_{00}(x)+A_{00}x^{m_{00}^+}-\hat{V}_{10}(x)+\min \{c_{01}, c_{02}\}.
\]

A straightforward computation shows that $h_{1}''(x)>0$, hence $h_{1}'(x)$ is increasing. Moreover, $h_{1}'(x_{00})=0$, so $h_{1}'(x)<0$ for $x<x_{00}$. Since $h_{1}(x_{00})=0$, we deduce that $h_{1}(x)>0$ for $x<x_{00}$, hence
\[
v_{00}(x)-M_{00}[v](x)>0.
\]

Next,
\[
v_{00}(x)-\mathcal{N}_{00}[v](x)
= v_{00}(x)-\min (v_{01}(x)+\chi_{01},\, v_{02}(x)+\chi_{02})
= -\min(\chi_{01},\chi_{02})<0.
\]

Thus, both obstacle inequalities are satisfied, and by Lemma \ref{lem:lemma3}, $v_{00}$ solves the QVI.

\medskip
\noindent
$\bullet$ For $x\geq x_{00}$, we have $v_{00}(x)=\hat{V}_{10}(x)-\min(c_{01},c_{02})$, hence
\[
v_{00}(x)-M_{00}[v](x)=0.
\]

Moreover, a straightforward calculus shows that
\[
\mathcal{H}[v_{00}](x)
= \frac{K_{10}-K_{00}}{K_{00}}x^\gamma - c_{02} >0.
\]

Next,
\[
v_{00}(x)-N_{00}[v](x)
= -\min(\chi_{01},\chi_{02})<0.
\]

Hence, $v_{00}$ satisfies the QVI on this region as well.

\medskip

\noindent
\textbf{Step 2. Verification for $v_{10}$.}

We have $\mathcal{H}[v_{10}](x)=0$ for all $x>0$.

Moreover,
\[
v_{10}(x)-M_{10}[v](x)
= v_{10}(x)-\max(v_{00}(x)-c_{10},\, v_{20}(x)-c_{12}).
\]

\noindent
$\bullet$ For $x<x_{00}$, define
\[
h_2(x)=v_{10}(x)-v_{00}(x)+c_{10}.
\]

Then
\[
h_2(x)=(K_{10}-K_{00})x^\gamma - A_{00}x^{m_{00}^+}+c_{10}.
\]

Since $h_{2}''(x)<0$, $h_{2}'$ is decreasing.

Using
\[
h_{2}'(x_{00})=0,
\]
we obtain
\[
h_{2}'(x)>0 \text{ for } x<x_{00},
\]
and hence $h_{2}$ is increasing.

Since $h_{2}(0)=c_{10}>0$, we deduce that
\[
v_{10}(x)-M_{10}[v](x)>0.
\]

\noindent
$\bullet$ For $x\geq x_{00}$, we have
\[
v_{10}(x)-M_{10}[v](x)
=\min(\min\{c_{01},c_{02}\}+c_{10},c_{12})>0.
\]

Finally,
\[
v_{10}(x)-N_{10}[v](x)
= -\min(\chi_{01},\chi_{02})<0.
\]

Thus, $v_{10}$ solves the QVI.
\hfill $\Box$
\epf

\bethe
Case where $K_{i0}=Cst_1 < K_{ij}=Cst_2$, $\; i=0,1,2,\; j=1,2$.
\\
\\
Let $v_{ij}(x)$ be defined as follows:
\begin{flalign}
    v_{i0}(x)&= \hat{V}_{i0}(x), \; i=0,1,2\\
    v_{ij}(x)&= (\hat{V}_{ij}(x)+A_{ij} x^{m_{ij}^+})\mathbf{1}_{x<x_{ij}}
    + (\hat{V}_{i0}(x) + \chi_{j0})\mathbf{1}_{x\geq x_{ij}},
    \; i=0,1,2,\; j=1,2.
\end{flalign}

Where $x_{ij}$ and $A_{ij}$ are obtained using $\mathrm{C}^1$ property of the value functions and are defined by
\begin{equation*}
(K_{ij}-K_{i0})(x_{ij})^{\gamma}
= \frac{m_{ij}^+}{m_{ij}^+ - \gamma}\chi_{j0},
\quad
A_{ij} = (K_{i0}-K_{ij})\frac{\gamma}{m_{ij}^+}(x_{ij})^{\gamma-m_{ij}^+}.
\end{equation*}

Such $v_{ij}(x)$ are the solutions associated to the QVI.

With:
\[
S^\xi_{i0} =\emptyset,\quad
C^\xi_{i0} = (0;+\infty),\quad
S^\eta_{i0} =\emptyset,\quad
C^\eta_{i0} = (0;+\infty),
\quad i=0,1,2
\]

\[
S^\xi_{ij} =\emptyset,\quad
C^\xi_{ij} = (0;+\infty),\quad
S^\eta_{ij} =[x_{ij}, \infty ),\quad
C^\eta_{ij} = (0;x_{ij}),
\quad i=0,1,2,\quad j=1,2.
\]

We can see that, by design,
\[
v_{01}(x)=v_{11}(x)=v_{21}(x)
\]
and
\[
v_{02}(x)=v_{12}(x)=v_{22}(x).
\]
\eethe

\bpf

We only verify the QVI for $v_{00}$ and $v_{20}$, since the remaining cases follow by exactly the same arguments.

\medskip

\noindent
\textbf{Step 1. Verification for $v_{00}$.}

Since
\[
v_{00}(x)=\hat V_{00}(x),
\]
we immediately have
\[
\mathcal H[v_{00}](x)=0.
\]

Furthermore,
\[
v_{00}(x)-M_{00}[v]
=
v_{00}(x)
-
\max\{v_{10}(x)-c_{01},\,v_{20}(x)-c_{02}\}.
\]

Given that
\[
v_{10}(x)=v_{20}(x),
\]
then
\[
v_{00}(x)-M_{00}[v]
=
v_{00}(x)-v_{10}(x) + \min \{ c_{01}, c_{02} \}
=
\min \{ c_{01}, c_{02} \} >0.
\]

This proves
\[
v_{00}(x)-M_{00}[v]>0.
\]

Next,
\[
v_{00}(x)-N_{00}[v]
=
v_{00}(x)
-
\min\{v_{01}(x)+\chi_{01},\,v_{02}(x)+\chi_{02}\}
=
-\min\{\chi_{01},\chi_{02}\}<0.
\]

Hence, $v_{00}$ satisfies the corresponding quasi-variational inequality.\\

\no
\textbf{Step 2. Verification for $v_{11}$.}
\\

\noindent
$\bullet$ For $x<x_{11}$, by construction,
\[
\mathcal H[v_{11}](x)=0.
\]

Moreover,
\[
v_{11}(x)-M_{11}[v]
=
v_{11}(x)-\max\{v_{01}(x)-c_{10},\,v_{21}(x)-c_{12}\}
=\min\{c_{10},c_{12}\}>0.
\]

Next,
\begin{flalign*}
v_{11}(x)-N_{11}[v]
&= v_{11}(x) - \min \{v_{10}(x) + \chi_{10}; v_{12}(x) + \chi_{12} \}.
\end{flalign*}

Let us determine
\[
m(x)=\min \{v_{10}(x) + \chi_{10}; v_{12}(x) + \chi_{12} \}.
\]

$x_{11}$ and $x_{12}$ are related to different switching costs $\chi_{10}$ and $\chi_{20}$, so they are independent.\\
\noindent\\
Assume that $x_{11} < x_{12}$:

Let
\[
h_3(x)=\hat V_{12}(x) + A_{12}x^{m_{12}^+}-\hat V_{10}(x)-\chi_{20},
\]
and
\begin{flalign*}
h_4(x)&=v_{12}(x) + \chi_{12} -(v_{10}(x) + \chi_{10})\\
&=\hat V_{12}(x) + A_{12}x^{m_{12}^+}
+\chi_{12}-\hat V_{10}(x)-\chi_{10}.
\end{flalign*}

We can see that
\begin{equation}
\label{id}
h_{4}(x)=h_3(x)+ \chi_{20}+ \chi_{12}- \chi_{10}.
\end{equation}

A straightforward calculus shows that $h_{4}''(x)<0$.

Hence, $h_{4}'(x)$ is decreasing. Since
\[
h_{4}'(x_{12})=0,
\]
it follows that $h_{4}(x)$ is increasing.

We have, from \ref{id},
\[
h_{4}(0)=\chi_{12} -\chi_{10}
\quad\text{and}\quad
h_{4}(x_{12})=\chi_{20}+\chi_{12} -\chi_{10}.
\]

$\star$ \quad \text{If } $\chi_{12} >\chi_{10}$, \text{ then } $h_{4}(x)>0$ for $x< x_{12}$, hence
\[
m(x)= v_{10}(x) + \chi_{10}.
\]

In this case,
\begin{flalign*}
v_{11}(x)-N_{11}[v]
&=v_{11}(x)-v_{10}(x) - \chi_{10}\\
&=\hat V_{11}(x) + A_{11}x^{m_{11}^+}-\hat V_{10}(x)-\chi_{10}.
\end{flalign*}

Let
\[
h_{5}(x)=\hat V_{11}(x) + A_{11}x^{m_{11}^+}-\hat V_{10}(x)-\chi_{10}.
\]

A straightforward calculus shows that $h_{5}''(x)<0$, hence $h_{5}'(x)$ is decreasing.

Using
\[
h_{5}'(x_{11})=0,
\]
it follows that
\[
h_{5}'(x) > 0, \quad x<x_{11}.
\]

Hence, $h_{5}(x)$ is increasing.

By using
\[
h_{5}(x_{11})=0,
\]
we prove that
\[
v_{11}(x)-N_{11}[v]<0.
\]

$\star$ \quad \text{If } $\chi_{12} <\chi_{10}$, $\text{ there is a } x^* \text{ such that:}$

\[
m(x)=
\begin{cases}
\hat V_{12}(x) + A_{12}x^{m_{12}^+} + \chi_{12}, &x < x^*, \\
v_{10}(x) + \chi_{10}, &x \geq x^*.
\end{cases}
\]

Then, for $x<x^*$,

\begin{flalign*}
h_{6}(x)&=V_{11}(x) + A_{11}x^{m_{11}^+} -V_{12}(x) - A_{12}x^{m_{12}^+}  - \chi_{12}\\
&= (A_{11}-A_{12})x^{m_{11}^+}- \chi_{12}
\end{flalign*}

\[
h_{6}(x)=-h_{4}(x)+\hat V_{11}(x)
+A_{11}x^{m_{11}^+}-\hat V_{10}(x)-\chi_{10}.
\]

It is easy to see that
\[
h_{7}(x)=\hat V_{11}(x) + A_{11}x^{m_{11}^+}-\hat V_{10}(x)-\chi_{10}
\]
is increasing for $x< x_{11}$.

Using
\[
h_{7}(0)=-\chi_{10}
\quad\text{and}\quad
h_{7}(x_{11})=0,
\]
we obtain that
\[
h_{7}(x) <0
\]
for $x< x_{11}$. This implies that $h_{7}(x^*)<0$.

Hence, $h_{6}(x)<0$ for $x<x^*$, which proves
\[
v_{11}(x)-N_{11}[v]<0.
\]

For $x\geq x^*$, the proof for $v_{11}(x)-N_{11}[v]<0$ is identical to that for
\[
\{\chi_{12} >\chi_{10}\} \cap \{x_{11} < x_{12}\}.
\]
\noindent\\
Assume now that $x_{12} \leq x_{11}$:

Then we have
\begin{flalign*}
m(x)&=\hat V_{10}(x)+\min\{ \chi_{10}, \chi_{12}+\chi_{20}\}
\quad \text{for } x_{12} \leq x \leq x_{11}\\
&=\hat V_{10}(x)+\chi_{10}.
\end{flalign*}

Hence, the proof for $v_{11}(x)-N_{11}[v]<0$ is identical to that for
\[
\{\chi_{12} >\chi_{10}\} \cap \{x_{11} < x_{12}\}.
\]

For $x<x_{12}$:

$\star$ \quad \text{If } $\chi_{12} >\chi_{10}$,
\[
m(x)= v_{10}(x) + \chi_{10}.
\]

And the proof for $v_{11}(x)-N_{11}[v]<0$ is identical to that for
\[
\{\chi_{12} >\chi_{10}\} \cap \{x_{11} < x_{12}\}.
\]

$\star$ \quad \text{If } $\chi_{12} <\chi_{10}$, there is some $x^*$ such that:

\[
m(x)=
\begin{cases}
\hat V_{12}(x) + A_{12}x^{m_{12}^+} + \chi_{12}, &x < x^*,\\
v_{10}(x) + \chi_{10}, &x \geq x^*.
\end{cases}
\]

The proof follows the same arguments as for
\[
\{\chi_{12} <\chi_{10}\} \cap \{x_{11} < x_{12}\}.
\]
\\
\\
\noindent
Therefore, the QVI is satisfied on $(0,x_{11})$.
\\
\\
\noindent
$\bullet$ For $x\ge x_{11}$, by construction,
\begin{flalign*}
v_{11}(x)-M_{11}[v]
&=v_{11}(x)-\max\{v_{01}(x)-c_{10},\,v_{21}(x)-c_{12}\}\\
&=\hat V_{10}(x)+ \chi_{10}
-\max\{v_{01}(x)-c_{10},\,v_{21}(x)-c_{12}\}\\
&=\hat V_{10}(x)+ \chi_{10}
-\max\{\hat{V}_{00}(x) +\chi_{10} -c_{10},
\hat{V}_{10}(x) +\chi_{10} -c_{12}\}\\
&=\min\{c_{10},c_{12}\} >0.
\end{flalign*}

Moreover,
\[
\mathcal H[v_{11}](x)
=
\frac{K_{10}-K_{11}}{K_{11}}x^\gamma+ \chi_{10},
\]
is negative by a straightforward calculus.

Hence,
\[
\min\big[\mathcal{H}[v_{11}](x),\, v_{11}(x)-M_{11}[v](x)\big]<0.
\]

Next,
\begin{flalign*}
v_{11}(x)-N_{11}[v]
&= v_{11}(x) - \min \{v_{10}(x) + \chi_{10}; v_{12}(x) + \chi_{12} \}.
\end{flalign*}

If $x_{12}<x_{11}$,
\[
v_{11}(x)-N_{11}[v]
= \hat{V}_{10}(x) +\chi_{10}
-\min \{\hat{V}_{10}(x) +\chi_{10};
\hat{V}_{10}(x) +\chi_{20} + \chi_{12} \}
=0.
\]

\noindent
Else, for $x_{11}<x<x_{12}$,
\[
v_{11}(x)-N_{11}[v]
= \hat{V}_{10}(x) +\chi_{10}
-\min \{\hat{V}_{10}(x) +\chi_{10};
\hat V_{12}(x) + A_{12}x^{m_{12}^+} + \chi_{12} \}.
\]

Define
\[
h_9(x)=\hat V_{12}(x) + A_{12}x^{m_{12}^+}
+ \chi_{12}-\hat{V}_{10}(x) -\chi_{10}.
\]

A straightforward calculus shows that $h_9$ is increasing and
\[
h_9(0)=\chi_{12}+\chi_{20}-\chi_{10}>0.
\]

Hence,
\[
\min \{\hat{V}_{10}(x) +\chi_{10};
\hat V_{12}(x) + A_{12}x^{m_{12}^+} + \chi_{12} \}
=\hat{V}_{10}(x) +\chi_{10}.
\]

For $x>x_{12}$,

\begin{flalign*}
\min \{v_{10}(x) + \chi_{10}; v_{12}(x) + \chi_{12} \}
&=\min \{\hat{V}_{10}(x) +\chi_{10};
\hat{V}_{10}(x) +\chi_{20} + \chi_{12} \}\\
&=\hat{V}_{10}(x) +\chi_{10}.
\end{flalign*}

In all cases,
\[
v_{11}(x)-N_{11}[v]=0.
\]

Hence, $v_{11}$ satisfies the QVI on $(x_{11},\infty)$, and therefore on $(0,\infty)$.
\\
\noindent
The proofs for $v_{01}$, $v_{02}$, $v_{10}$, $v_{12}$, $v_{21}$, and $v_{22}$ are completely analogous and are therefore omitted.
\hfill $\Box$
\epf

\bethe
Case where
\[
K_{ij}=Cst_1<K_{2j}=Cst_2,\qquad i=0,1,\quad j=0,1,2.
\]

Let the candidate value functions be defined by
\begin{align}
v_{ij}(x)&=
\bigl(\hat V_{ij}(x)+A_{ij}x^{m_{ij}^+}\bigr)\mathbf 1_{x<x_{ij}}
+\bigl(\hat V_{2j}(x)-c_{i2}\bigr)\mathbf 1_{x\ge x_{ij}},
\qquad i=0,1,\quad j=0,1,2,\\
v_{2j}(x)&=\hat V_{2j}(x),\qquad j=0,1,2.
\end{align}

Where $x_{ij}$ and $A_{ij}$ are defined by
\begin{equation*}
(K_{2j}-K_{ij})(x_{ij})^{\gamma}
= \frac{m_{ij}^+}{m_{ij}^+ - \gamma} c_{i2},
\quad
A_{ij} = (K_{2j}-K_{ij})\frac{\gamma}{m_{ij}^+}(x_{ij})^{\gamma-m_{ij}^+}.
\end{equation*}

Such $v_{ij}(x)$ are the solutions associated to the QVI.

With:
\[
S^\xi_{2j} =\emptyset,\quad
C^\xi_{2j} = (0;+\infty),\quad
S^\eta_{2j} =\emptyset,\quad
C^\eta_{2j} = (0;+\infty),
\quad j=0,1,2
\]

\[
S^\xi_{ij} =[x_{ij}, \infty ),\quad
C^\xi_{ij} = (0;x_{ij}),\quad
S^\eta_{ij}=\emptyset,\quad
C^\eta_{ij}=(0;+\infty),
\quad i=0,1,\quad j=0,1,2.
\]
\eethe

\bpf

We only verify the QVI for $v_{00}$ and $v_{20}$, since the remaining cases follow by exactly the same arguments.

\medskip

\noindent
\textbf{Step 1. Verification for $v_{00}$.}

\noindent
$\bullet$ For $x<x_{00}$, by construction,
\[
\mathcal H[v_{00}](x)=0.
\]

Moreover,
\[
v_{00}(x)-M_{00}[v]
=
v_{00}(x)-\max\{v_{10}(x)-c_{01},\,v_{20}(x)-c_{02}\}.
\]

Let us determine
\[
n(x)=\max \{v_{10}(x) - c_{01}, v_{20}(x) - c_{02} \}.
\]

$x_{00}$ and $x_{10}$ are related to different switching costs $c_{02}$ and $c_{12}$, so they are independent.\\
\noindent\\
Assume that $x_{00} < x_{10}$:

Let
\[
h_{10}(x)=\hat V_{10}(x) + A_{10}x^{m_{10}^+}-\hat V_{20}(x)+c_{12},
\]
and
\begin{flalign*}
h_{11}(x)&=v_{10}(x) - c_{01} -(v_{20}(x) - c_{02})\\
&=\hat V_{10}(x) + A_{10}x^{m_{10}^+}
-c_{01} -\hat V_{20}(x) +c_{02}.
\end{flalign*}

We can see that
\begin{equation}
\label{id2}
h_{11}(x)=h_{10}(x)- c_{01}+ c_{02} - c_{12}.
\end{equation}

A straightforward calculus shows that $h_{11}''(x)>0$.

Hence, $h_{11}'(x)$ is increasing. Since
\[
h_{11}'(x_{10})=0,
\]
$h_{11}(x)$ is decreasing.

We have, from \ref{id2},
\[
h_{11}(0)= c_{02}- c_{01}
\quad\text{and}\quad
h_{11}(x_{10})=- c_{01}+ c_{02} - c_{12} <0.
\]

$\star$ \quad \text{If } $c_{02}< c_{01}$, \text{ then } $h_{11}(x)<0$ for $x< x_{00}$, hence
\[
n(x)=\hat V_{20}(x) - c_{02}.
\]

In this case,

Let
\[
h_{12}(x)=\hat V_{00}(x) + A_{00}x^{m_{00}^+}
-\hat V_{20}(x) + c_{02}.
\]

A straightforward calculus shows that $h_{12}''(x)>0$. Hence, $h_{12}'(x)$ is increasing.

Using
\[
h_{12}'(x_{00})=0,
\]
it follows that $h_{12}'(x)<0$. Hence, $h_{12}(x)$ is decreasing.

Given that
\[
h_{12}(x_{00})=0,
\]
we get the desired result.

$\star$ \quad \text{If } $c_{02}> c_{01}$, $\text{ there is a } x^* \text{ such that:}$

\[
n(x)=
\begin{cases}
v_{20}(x) - c_{02}, &x\geq x^*,\\
\hat V_{10}(x) + A_{10}x^{m_{10}^+}- c_{01}, &x < x^*.
\end{cases}
\]

Then, for $x \geq x^*$, the proof for $v_{00}(x)-M_{00}[v]>0$ is identical to that for
\[
\{c_{02}< c_{01}\} \cap \{x_{00} < x_{10}\}.
\]

For $x < x^*$,
\noindent
Let
\begin{flalign*}
h_{13}(x) &=\hat V_{00}(x) + A_{00}x^{m_{00}^+}
-\hat V_{10}(x) - A_{10}x^{m_{10}^+} + c_{10}\\
&=(A_{00} -A_{10})x^{m_{00}^+}+ c_{10}.
\end{flalign*}

If $(A_{00} -A_{10})>0$, we immediately get the desired result.

Else,
$h_{13}'(x)<0$ and hence $h_{13}(x)$ is decreasing.

Let
\[
h_{14}(x)=\hat V_{10}(x) + A_{10}x^{m_{10}^+}
-\hat V_{20}(x) -c_{10}+ c_{02}.
\]

We can see that
\begin{flalign*}
h_{13}(x)&=-h_{14}(x)+\hat V_{00}(x)
+A_{00}x^{m_{00}^+}-\hat V_{20}(x)+c_{02}\\
&=-h_{14}(x)+h_{12}(x).
\end{flalign*}

Since
\[
h_{13}(x^*)=h_{12}(x^*)>0,
\]
this proves that
\[
v_{00}(x)-M_{00}[v]>0.
\]

\noindent\\
Assume now that $x_{10} \leq x_{00}$:

Then we have
\begin{flalign*}
n(x)&=\hat V_{20}(x)-\min\{c_{01} + c_{12}, c_{02}\}
\quad \text{for } x_{10} \leq x \leq x_{00}\\
&=\hat V_{20}(x)-c_{02}.
\end{flalign*}

Hence, the proof for $v_{00}(x)-M_{00}[v]>0$ is identical to that for
\[
\{c_{02}< c_{01}\} \cap \{x_{00} < x_{10}\} \cap \{x< x^*\}.
\]

For $x<x_{10}$:

$\star$ \quad \text{If } $c_{02}> c_{01}$,
\[
n(x)= v_{10}(x) - c_{01}.
\]

And the proof for $v_{00}(x)-M_{00}[v]>0$ is identical to that for
\[
\{c_{02}< c_{01}\} \cap \{x_{00} < x_{10}\} \cap \{x< x^*\}.
\]

$\star$ \quad \text{If } $c_{02}< c_{01}$, there is some $x^*$ such that:

\[
n(x)=
\begin{cases}
v_{20}(x) - c_{02}, &x < x^*,\\
\hat V_{10}(x) + A_{10}x^{m_{10}^+}- c_{01}, &x \geq x^*.
\end{cases}
\]

The proof for $v_{00}(x)-M_{00}[v]>0$ follows the same arguments as for
\[
\{c_{02}< c_{01}\} \cap \{x_{00} < x_{10}\}.
\]
\\
\noindent
Next,
\[
v_{00}(x)-N_{00}[v]
=
v_{00}(x)-\min\{v_{01}(x)+\chi_{01},\,v_{02}(x)+\chi_{02}\}.
\]

Using the same reasoning as for $v_{00}(x)-M_{00}[v]>0$, we deduce that
\[
v_{00}(x)-N_{00}[v]<0.
\]
\\
\noindent
Therefore, the QVI is satisfied on $(0,x_{00})$.
\\
\\
\noindent
$\bullet$ For $x\geq x_{00}$, we have
\[
\mathcal H[v_{00}](x)
=
\frac{K_{20}-K_{00}}{K_{00}}x^\gamma- c_{02}.
\]

A straightforward calculus shows that $\mathcal H[v_{00}](x)$ is positive.

Moreover:
\begin{flalign*}
v_{00}(x)-M_{00}[v]
&=v_{00}(x)-\max\{v_{10}(x)-c_{01},\,v_{20}(x)-c_{02}\}.
\end{flalign*}

If $x_{00} > x_{10}$:
\begin{flalign*}
v_{00}(x)-M_{00}[v]
&=\hat V_{20}(x) -c_{02}
-\max\{\hat V_{20}(x)-c_{01}-c_{12},\,\hat V_{20}(x)-c_{02}\}\\
&=c_{02}-c_{02}=0.
\end{flalign*}

\indent
And
\begin{flalign*}
v_{00}(x)-N_{00}[v]
&=v_{00}(x) - \min \{v_{01}(x) + \chi_{01}; v_{02}(x) + \chi_{02}\}\\
&=-\max \{\chi_{01},\chi_{02}\}<0.
\end{flalign*}

\noindent
If $x_{00}<x_{10}$, we only need to handle the case where $x_{00}<x<x_{10}$:
\begin{flalign*}
v_{00}(x)-M_{00}[v]
&=\hat V_{20}(x) -c_{02}
-\max\{v_{10}(x)-c_{01},\,v_{20}(x)-c_{02}\}.
\end{flalign*}

\indent
It suffices to show that, on $(x_{00}, x_{10})$,
\[
n(x)=\max\{v_{10}(x)-c_{01},\,v_{20}(x)-c_{02}\}
=v_{20}(x)-c_{02}.
\]

From above, we know that if $c_{02}<c_{01}$, we are done.

Else,
\[
n(x)=
\begin{cases}
v_{20}(x) - c_{02}, &x\geq x^*,\\
\hat V_{10}(x) + A_{10}x^{m_{10}^+}- c_{01}, &x < x^*.
\end{cases}
\]

Let's show that $x^*<x_{00}$.
\begin{flalign*}
h_{11}(x)&=v_{10}(x) - c_{01} -(v_{20}(x) - c_{02})\\
&=\hat V_{10}(x) + A_{10}x^{m_{10}^+}
-c_{01} -\hat V_{20}(x) +c_{02}\\
&=\hat V_{10}(x) + A_{10}x^{m_{10}^+}
+\hat V_{00}(x) + A_{00}x^{m_{00}^+}\\
&\quad-(\hat V_{00}(x) + A_{00}x^{m_{00}^+})
-c_{01} -\hat V_{20}(x) +c_{02}\\
&=(A_{10} -A_{00})x^{m_{00}^+}
+\hat V_{00}(x) + A_{00}x^{m_{00}^+}
-c_{01} -\hat V_{20}(x) +c_{02}.
\end{flalign*}

$x_{00}<x_{10} \implies A_{10} < A_{00}$, hence
\[
(A_{10} -A_{00}){x_{00}}^{m_{00}^+} <0.
\]

On the other hand, $f(x_{00})=-c_{01}$, where
\[
f(x)=V_{00}(x) + A_{00}x^{m_{00}^+}
-c_{01} -\hat V_{20}(x) +c_{02}.
\]

We deduce that
\[
h_{11}(x_{00})<0.
\]

Using the fact that $h_{11}$ is decreasing on $(0, x_{10})$, and $h'_{11}(x^*)=0$, we get the desired result.

Finally, for $v_{00}(x)-N_{00}[v]$, we have:
\begin{flalign*}
v_{00}(x)-N_{00}[v]
&=v_{00}(x) - \min \{v_{01}(x) + \chi_{01}; v_{02}(x) + \chi_{02}\}\\
&=-\max \{\chi_{01},\chi_{02}\}<0.
\end{flalign*}

Hence, we deduce that $v_{00}$ satisfies the QVI on $(x_{00},\infty)$, and therefore on $(0,\infty)$.

\noindent
\textbf{Step 2. Verification for $v_{20}$.}

Since
\[
v_{20}(x)=\hat V_{20}(x),
\]
\indent we immediately have
\[
\mathcal H[v_{20}](x)=0.
\]

Furthermore,
\[
v_{20}(x)-M_{20}[v]
=
v_{20}(x)
-
\max\{v_{00}(x)-c_{20},\,v_{10}(x)-c_{21}\}.
\]

\noindent
Assume that
\[
\max\{v_{00}(x)-c_{20},\,v_{10}(x)-c_{21}\}
= v_{00}(x)-c_{20}.
\]

\noindent
$\bullet$ For $x<x_{00}$:
\[
v_{20}(x)-M_{20}[v]
=\hat V_{20}(x) -\hat V_{00}(x)
-A_{00}x^{m_{00}^+}+ c_{20}.
\]

Let
\[
h_{15}(x)=\hat V_{20}(x) -\hat V_{00}(x)
-A_{00}x^{m_{00}^+}+ c_{20}.
\]

A straightforward calculus shows that $h_{15}(x)$ is increasing on $(0, x_{00})$. Given that $h_{15}(0)=c_{20}>0$, we deduce that
\[
v_{20}(x)-M_{20}[v]>0.
\]

\noindent
$\bullet$ For $x\geq x_{00}$:
\begin{align*}
v_{20}(x)-M_{20}[v]
&=v_{20}(x) -\{(v_{20}(x)- c_{02}) -c_{20}\}\\
&=c_{02}+c_{20}>0.
\end{align*}

We deduce that
\[
v_{20}(x)-M_{20}[v]>0.
\]

\noindent
Else, assume that
\[
\max\{v_{00}(x)-c_{20},\,v_{10}(x)-c_{21}\}
= v_{10}(x)-c_{21}.
\]

\noindent
For $x<x_{10}$:
\[
v_{20}(x)-M_{20}[v]
=\hat V_{20}(x) -\hat V_{10}(x)
-A_{10}x^{m_{10}^+}+ c_{21}.
\]

Let
\[
h_{16}(x)=\hat V_{20}(x) -\hat V_{10}(x)
-A_{10}x^{m_{10}^+}+ c_{21}.
\]

A straightforward calculus shows that $h_{16}(x)$ is increasing on $(0, x_{00})$. Given that $h_{16}(0)=c_{21}>0$, we deduce that
\[
v_{20}(x)-M_{20}[v]>0.
\]

\noindent
For $x \geq x_{10}$:
\begin{align*}
v_{20}(x)-M_{20}[v]
&=v_{20}(x) -\{(v_{20}(x)- c_{12}) -c_{21}\}\\
&=c_{12}+c_{21}>0.
\end{align*}

We deduce that
\[
v_{20}(x)-M_{20}[v]>0.
\]

\noindent
Finally,
\begin{flalign*}
v_{20}(x)-N_{20}[v]
&=v_{20}(x)-\min\{v_{21}(x)+\chi_{01}, v_{22}(x)+\chi_{22}\}\\
&=v_{20}(x)-v_{21}(x) - \max \{\chi_{01}, \chi_{02}\}\\
&=-\max \{\chi_{01}, \chi_{02}\}<0.
\end{flalign*}

Hence, $v_{20}$ satisfies the corresponding quasi-variational inequality on $(0, +\infty)$.

\noindent
The proofs for $v_{01}$, $v_{02}$, $v_{10}$, $v_{11}$, and $v_{12}$ are completely analogous and are therefore omitted.

The conclusion follows from Lemma~\ref{lem:lemma3}.
\hfill $\Box$
\epf
\bethe
Case where
\[
K_{ij}=Cst_1<K_{i2}=Cst_2,\qquad i=0,1,2,\quad j=0,1,
\]
\begin{align}
v_{ij}(x)&=\hat V_{ij}(x),\quad i=0,1,2,\quad j=0,1,\\
v_{i2}(x)&=\bigl(\hat V_{i2}(x)+A_{i2}x^{m_{i2}^+}\bigr)\mathbf 1_{\{x<x_{i2}\}}
+\bigl(v_{i0}(x)+\min\{\chi_{20},\chi_{21}\}\bigr)\mathbf 1_{\{x\ge x_{i2}\}},\quad i=0,1,2.
\end{align}

where $x_{i2}$ and $A_{i2}$ are defined by
\begin{equation*}
(K_{i0}-K_{2j})(x_{i2})^{\gamma}
= \frac{m_{i2}^+}{m_{i2}^+-\gamma}c_{i2},
\quad
A_{i2}=(K_{i0}-K_{2j})\frac{\gamma}{m_{i2}^+}(x_{i2})^{\gamma-m_{ij}^+}.
\end{equation*}
\eethe

Such $v_{ij}(x)$ are the candidate solutions associated with the QVI system.

With:
\[
S^\xi_{ij}=\emptyset,\quad C^\xi_{ij}=(0,+\infty),\quad
S^\eta_{ij}=\emptyset,\quad C^\eta_{ij}=(0,+\infty),\quad j=0,1,2,
\]

\[
S^\xi_{ij}=[x_{i2},\infty),\quad C^\xi_{i2}=(0,x_{i2}),\quad
S^\eta_{i2}=\emptyset,\quad C^\eta_{i2}=(0,+\infty),\quad
i=0,1,\quad j=0,1,2.
\]

\bpf

We only verify the QVI for $v_{00}$ and $v_{02}$, since the remaining cases follow by exactly the same arguments.

\noindent
\textbf{Step 1. Verification for $v_{00}$.}

Since
\[
v_{00}(x)=\hat V_{00}(x),
\]
we immediately have
\[
\mathcal H[v_{00}](x)=0.
\]

Furthermore,
\begin{flalign*}
v_{00}(x)-M_{00}[v]
&=v_{00}(x)-\max\{v_{10}(x)-c_{01},v_{20}(x)-c_{02}\}.\\
&=v_{00}(x)-v_{10}(x)+\min\{c_{01},c_{02}\}\\
&=\min\{c_{01},c_{02}\}>0.
\end{flalign*}

Finally,
\[
v_{00}(x)-N_{00}[v]
=
v_{00}(x)-\min\{v_{01}(x)+\chi_{01},v_{02}(x)+\chi_{02}\}.
\]

If the minimum is attained at
\[
v_{01}(x)+\chi_{01},
\]
then
\[
v_{00}(x)-N_{00}[v]
=
v_{00}(x)-v_{01}(x)-\chi_{01}
=
-\chi_{01}<0.
\]

Otherwise, for $x<x_{02}$,
\[
v_{00}(x)-N_{00}[v]
=
v_{00}(x)-\hat V_{02}(x)-A_{02}x^{m_{02}^+}-\chi_{02}.
\]

Define
\[
G(x)=\hat V_{00}(x)-\hat V_{02}(x)-A_{02}x^{m_{02}^+}-\chi_{02}.
\]

A direct computation shows that
\[
G''(x)>0.
\]
Hence, $G'$ is increasing. Since
\[
G'(x_{02})=0,
\]
it follows that
\[
G'(x)<0,\qquad x<x_{02}.
\]

Using
\[
G(x_{02})=-\min\{\chi_{20},\chi_{21}\}-\chi_{02}<0,
\qquad
G(0)=-\chi_{02}<0,
\]
we conclude that
\[
G(x)<0,\qquad x<x_{02},
\]
which proves
\[
v_{00}(x)-N_{00}[v]<0.
\]

For $x>x_{02}$,
\[
v_{00}(x)-N_{00}[v]
=
v_{00}(x)-V_{00}(x)-\chi_{02}
-\min\{\chi_{20},\chi_{21}\}<0.
\]

Hence, $v_{00}$ satisfies the corresponding quasi-variational inequality.

\noindent
\textbf{Step 2. Verification for $v_{02}$.}
\\

\noindent
$\bullet$ For $x<x_{02}$, by construction,
\[
\mathcal H[v_{02}](x)=0.
\]

Moreover,
\begin{flalign*}
v_{02}(x)-M_{02}[v]
&=
v_{02}(x)-\max\{v_{12}(x)-c_{01},\,v_{22}(x)-c_{02}\}.\\
&=\min\{c_{01},c_{02}\}>0.
\end{flalign*}

Next,
\begin{flalign*}
v_{02}(x)-N_{02}[v]
&=
v_{02}(x)
-\min\{v_{01}(x)+\chi_{21},\,v_{00}(x)+\chi_{20}\},\\
&=
v_{02}(x)-v_{01}(x)-\min\{\chi_{21},\chi_{20}\},\\
&=
\hat V_{02}(x)+A_{02}x^{m_{02}^+}-\hat V_{01}(x)
-\min\{\chi_{21},\chi_{20}\}.
\end{flalign*}

Let
\[
h_{17}(x)=\hat V_{02}(x)+A_{02}x^{m_{02}^+}-\hat V_{01}(x)
-\min\{\chi_{21},\chi_{20}\}.
\]

$h_{17}'(x)>0$, hence $h_{17}$ is increasing.

Using
\[
h_{17}(x_{02})=0,
\]
we deduce that
\[
v_{02}(x)-N_{02}[v]<0.
\]

Therefore, the QVI is satisfied on $(0,x_{02})$.
\\

\noindent
$\bullet$ For $x\ge x_{02}$,
\[
\mathcal H[v_{02}](x)
=
\frac{K_{00}-K_{02}}{K_{02}}x^\gamma
+\min\{\chi_{21},\chi_{20}\},
\]
which is negative by definition of the switching threshold.

Hence,
\[
\min\big[\mathcal{H}[v_{02}](x),\,v_{02}(x)-M_{02}[v](x)\big]<0.
\]

Finally,
\begin{flalign*}
v_{02}(x)-N_{02}[v]
&=v_{02}(x)-\min\{v_{01}(x)+\chi_{21},\,v_{00}(x)+\chi_{20}\}.\\
&=v_{02}(x)-v_{01}(x)-\min\{\chi_{21},\chi_{20}\}.\\
&=\hat V_{01}(x)+\min\{\chi_{21},\chi_{20}\}
-\hat V_{01}(x)-\min\{\chi_{21},\chi_{20}\}=0.
\end{flalign*}

Hence, $v_{02}$ satisfies the QVI on $(x_{02},\infty)$, and therefore on $(0,\infty)$.\hfill $\Box$

\epf

\section{A Numerical Procedure for Computing the Value Functions}

In this section, we propose a numerical procedure for computing the value functions when the qualitative structure of the switching regions is known, while the numerical values of the thresholds are not explicitly available.

\subsection{Notation and Structure of the Switching Regions}

Let us first introduce the following notation. For $i,j,k,l$ denoting the corresponding regimes, define
\begin{align*}
S^{\xi,j}_{ik}
&:= \text{the region where Player I switches from regime $i$ to regime $k$,while Player II's current regime is $j$,}\\
S^{\eta,i}_{jl}
&:= \text{the region where Player II switches from regime $j$ to regime $l$, while Player I's current regime is $i$.}
\end{align*}

To illustrate our approach, we restrict attention to the particular case in which $(\mu_0,\nu_0)=(0,0)$ and the switching regions have the following threshold structure:
\begin{align}
S^{\xi,0}_{01}
&=(0,x^{0}_{01}],&
S^{\xi,0}_{10}
&=[x^{0}_{10},+\infty),&
\\
S^{\eta,0}_{01}
&=(0,y^{0}_{01}],&
S^{\eta,0}_{12}
&=[y^{0}_{12},+\infty).
\end{align}

All other switching regions are assumed to be empty.

Without loss of generality, we assume that the thresholds satisfy
\begin{equation}
x^{0}_{01}
<
y^{0}_{01}
<
x^{0}_{10}
<
y^{0}_{12}.
\end{equation}

\subsection{Auxiliary Functionals}

For $a,b,x>0$, let
\begin{align}
R_1(x,a)
&:= \mathbb{E}\left[e^{-r\tau_a}\right],\\
R_2(x,a,b)
&:= \mathbb{E}\left[e^{-r\tau_{ab}}\right],\\
R_3(x,a,b)
&:= \mathbb{E}\left[
e^{-r\tau_a}\mathbbm{1}_{\{\tau_a<\tau_b\}}
\right].
\end{align}
Where $\tau_a$ and $\tau_b$ denote the corresponding hitting times.

We also introduce the following expected discounted running-payoff functionals:
\begin{align}
F_1(a)
&:= \mathbb{E}\left[
\int_0^{\tau_a} e^{-rs}f(X_s^x)\,ds
\right],\\
F_2(a,b)
&:= \mathbb{E}\left[
\int_{\tau_a}^{\tau_b}e^{-rs}f(X_s^x)\,ds
\right],\\
F_3(a,b)
&:= \mathbb{E}\left[
\int_0^{\tau_a\wedge\tau_b}
e^{-rs}f(X_s^x)\,ds
\right].
\end{align}

The computation of these expectation functionals can be found in the Appendix of \cite{PVT}.

\subsection{Case $x<x^{0}_{01}$}

Assume that
\[
x<x^{0}_{01}.
\]

In this case, the optimal strategy can be described as follows:

\paragraph{Initial regime $(\mu_0,\nu_0)=(0,0)$.}

Both players immediately switch from regime $0$ to regime $1$, paying the corresponding switching costs $c_{01}$ and $\chi_{01}$, respectively. The problem then reduces to the determination of the optimal strategy associated with $J_{11}$. But there is no switching in the state $(\mu_0,\nu_0)=(1,1)$, so the optimal strategy is known.

\[
J_{00}(x)=-c_{01}+\chi_{01}+J_{11}(x),
\qquad
\text{and then}
\qquad
v_{00}(x)=-c_{01}+\chi_{01}+\hat V_{11}(x).
\]

\subsection{Case $x^{0}_{01}<x<y^{0}_{01}$}

Assume that
\[
x^{0}_{01}<x<y^{0}_{01}.
\]

In this case, the optimal strategy can be described as follows.

\paragraph{Initial regime $(\mu_0,\nu_0)=(0,0)$.}

Player 2 has to immediately switch to regime $1$, paying the corresponding switching cost $\chi_{01}$. Then, he lets the process diffuse until it hits $y^0_{12}$ at $\tau^0_{y_{12}}$; he then switches to regime $2$, paying $\chi_{12}$. But there is no switching in the state $(\mu_0,\nu_0)=(0,2)$, so the optimal strategy is known.

\[
J_{00}(x)=\mathbb{E}\bigg[
\integ{0}{\tau^0_{y_{12}}}e^{-rt}f(X^x_t)\,dt
+
e^{-r\tau^0_{y_{12}}}
\bigg(
\chi_{12}+J_{02}(x)
\bigg)
\bigg].
\]

There is no switching in the state $(\mu_0,\nu_0)=(0,2)$, so
$J_{02}(x)=\hat V_{02}(x)$.

We get
\[
v_{00}(x,y^{0}_{12})
=
\chi_{01}
+
F_1(y^{0}_{12})
+
\big(
\chi_{12}+J_{02}(y^{0}_{12})
\big)
R_1(x,y^{0}_{12}).
\]

Player 2 aims at minimizing the payoff, so
\[
v_{00}(x)
=
\min_{y^{0}_{12}}
\bigg[
\chi_{01}
+
F_1(y^{0}_{12})
+
\big(
\chi_{12}+J_{02}(y^{0}_{12})
\big)
R_1(x,y^{0}_{12})
\bigg].
\]

\subsection{Case $y^{0}_{01}<x<x^{0}_{10}$}

Assume that
\[
y^{0}_{01}<x<x^{0}_{10}.
\]

In this case, the optimal strategy can be described as follows.

\paragraph{Initial regime $(\mu_0,\nu_0)=(0,0)$.}

Let the process diffuse until it hits $y^0_{01}$ at $\tau^0_{y_{01}}$. Then, Player 2 switches to regime $1$, paying the corresponding switching cost $\chi_{01}$. Then, he lets the process diffuse until it hits $y^0_{12}$ at $\tau^0_{y_{12}}$; he then switches to regime $2$, paying $\chi_{12}$. But there is no switching in the state $(\mu_0,\nu_0)=(0,2)$, so the optimal strategy is known.

\[
J_{00}(x)=\mathbb{E}\bigg[
\integ{0}{\tau^0_{y_{01}}}e^{-rt}f(X^x_t)\,dt
+
e^{-r\tau^0_{y_{01}}}
\bigg(
\chi_{01}+J_{01}(y^0_{01})
\bigg)
\bigg].
\]

\[
J_{01}(y^0_{01})
=
\mathbb{E}\bigg[
\integ{0}{\tau^0_{y_{12}}}e^{-rt}f(X^x_t)\,dt
+
e^{-r\tau^0_{y_{12}}}
\bigg(
\chi_{12}+J_{02}(y^0_{12})
\bigg)
\bigg].
\]

Thus, the computation of $J_{01}$ is straightforward. Once we know $J_{01}$, we also know $J_{00}$.

\subsection{Case $x^{0}_{10}<x<y^{0}_{12}$}

Assume that
\[
x^{0}_{10}<x<y^{0}_{12}.
\]

In this case, the optimal strategy can be described as follows.

\paragraph{Initial regime $(\mu_0,\nu_0)=(0,0)$.}

Let the process diffuse until it hits $y^0_{01}$ at $\tau^0_{y_{01}}$. Then, Player 2 switches to regime $1$, paying the corresponding switching cost $\chi_{01}$. Then, he lets the process diffuse until it hits $y^0_{12}$ at $\tau^0_{y_{12}}$; he then switches to regime $2$, paying $\chi_{12}$. But there is no switching in the state $(\mu_0,\nu_0)=(0,2)$, so the optimal strategy is known.

\[
J_{00}(x)=\mathbb{E}\bigg[
\integ{0}{\tau^0_{y_{01}}}e^{-rt}f(X^x_t)\,dt
+
e^{-r\tau^0_{y_{01}}}
\bigg(
\chi_{01}+J_{01}(y^0_{01})
\bigg)
\bigg].
\]

\[
J_{01}(y^0_{01})
=
\mathbb{E}\bigg[
\integ{0}{\tau^0_{y_{12}}}e^{-rt}f(X^x_t)\,dt
+
e^{-r\tau^0_{y_{12}}}
\bigg(
\chi_{12}+J_{02}(y^0_{12})
\bigg)
\bigg].
\]

\subsection{Case $x>y^{0}_{12}$}

Assume that
\[
x>y^{0}_{12}.
\]

In this case, the optimal strategy can be described as follows.

\paragraph{Initial regime $(\mu_0,\nu_0)=(0,0)$.}

Let the process diffuse until it hits $y^0_{01}$ at $\tau^0_{y_{01}}$. Then, Player 2 switches to regime $1$, paying the corresponding switching cost $\chi_{01}$. Then, he lets the process diffuse until it hits $y^0_{12}$ at $\tau^0_{y_{12}}$; he then switches to regime $2$, paying $\chi_{12}$. But there is no switching in the state $(\mu_0,\nu_0)=(0,2)$, so the optimal strategy is known.

\[
J_{00}(x)=\mathbb{E}\bigg[
\integ{0}{\tau^0_{y_{01}}}e^{-rt}f(X^x_t)\,dt
+
e^{-r\tau^0_{y_{01}}}
\bigg(
\chi_{01}+J_{01}(y^0_{01})
\bigg)
\bigg].
\]

\[
J_{01}(y^0_{01})
=
\mathbb{E}\bigg[
\integ{0}{\tau^0_{y_{12}}}e^{-rt}f(X^x_t)\,dt
+
e^{-r\tau^0_{y_{12}}}
\bigg(
\chi_{12}+J_{02}(y^0_{12})
\bigg)
\bigg].
\]

\clearpage

\section{Graphic Illustration}

In this section, we shall provide illustrations of the players' strategies. We rely on the case where
$K_{ij}=Cst_1<K_{i2}=Cst_2,\quad i=0,1,2,\quad j=0,1$.
We show illustrations for $v_{02}$ and $v_{12}$. Suppose that the current regime of Player II is $R_2$, i.e., $\nu_0=2$, and the current regime of Player I is $R_0$, i.e., $\mu_0=0$, or $R_1$, i.e., $\mu_0=1$. The strategy of Player II depends on the sign of $\chi_{20}-\chi_{21}$, while the strategy of Player I is not to switch at all.

\begin{figure}[ht]
\centering

\begin{minipage}{0.48\textwidth}
\centering
\resizebox{\linewidth}{!}{%
\begin{tikzpicture}[>=Stealth,x=1cm,y=1cm,scale=0.8,transform shape]

\begin{scope}
    \draw[->] (0,0) -- (10,0);
    \draw[->] (0,0) -- (0,4) node[above] {$x$};

    \draw[blue,very thick] (0.5,0.3) -- (9.2,3.5);
    \draw[densely dashed] (0,2) -- (5,2);

    \draw[densely dashed] (5,0) -- (5,2);

    \node[left] at (0,2) {$x_{02}$};
\end{scope}

\begin{scope}[yshift=-3cm]

    \draw[densely dashed] (0,0) -- (10,0);

    \draw[densely dashed] (5,-0.8) -- (5,0.8);

    \draw[densely dashed] (0,0.8) -- (10,0.8) node[right] {$R_0$};
    \draw[red,very thick] (0,0.8) -- (10,0.8);

    \draw[densely dashed] (0,0) -- (10,0) node[right] {$R_1$};

    \draw[densely dashed] (0,-0.8) -- (10,-0.8) node[right] {$R_2$};

    \node[left] at (0,0) {Player 1};

\end{scope}

\begin{scope}[yshift=-6cm]

    \draw[densely dashed] (0,0) -- (10,0);

    \draw[green!60!black,very thick] (5,-0.8) -- (5,0.8);

    \draw[densely dashed] (0,-0.8) -- (10,-0.8) node[right] {$R_2$};
    \draw[green!60!black,very thick] (0,-0.8) -- (5,-0.8);
    \draw[green!60!black,very thick] (5,0.8) -- (10,0.8);

    \draw[densely dashed] (0,0.8) -- (10,0.8) node[right] {$R_0$};
    \draw[densely dashed] (6,0) -- (10,0) node[right] {$R_1$};

    \node[left] at (0,0) {Player 2};

\end{scope}

\end{tikzpicture}%
}
\caption{$\mu_0=0,\ \nu_0=2,\ \chi_{20}-\chi_{21}<0$}
\end{minipage}
\hfill
\begin{minipage}{0.48\textwidth}
\centering
\resizebox{\linewidth}{!}{%
\begin{tikzpicture}[
    >=Stealth,
    x=1cm,
    y=1cm,
    scale=0.8,
    transform shape
]

\begin{scope}
    \draw[->] (0,0) -- (10,0);
    \draw[->] (0,0) -- (0,4) node[above] {$x$};

    \node[left] at (0,2.4) {$x_{02}$};

    \draw[blue,very thick] (0.5,0.3) -- (9.2,3.5);

    \draw[densely dashed] (0,2.4) -- (6,2.4);
    \draw[densely dashed] (6,0) -- (6,2.4);

\end{scope}

\begin{scope}[yshift=-3cm]

    \draw[densely dashed] (0,0) -- (10,0);

    \draw[densely dashed] (6,-0.8) -- (6,0.8);

    \draw[densely dashed] (0,0.8) -- (10,0.8)
        node[right] {$R_0$};
    \draw[red,very thick] (0,0.8) -- (10,0.8);

    \draw[densely dashed] (0,0) -- (10,0)
        node[right] {$R_1$};

    \draw[densely dashed] (0,-0.8) -- (10,-0.8)
        node[right] {$R_2$};

    \node[left] at (0,0) {Player 1};

\end{scope}

\begin{scope}[yshift=-6cm]

    \draw[densely dashed] (0,0) -- (10,0);

    \draw[green!60!black,very thick] (6,-0.80) -- (6,0);

    \draw[densely dashed] (0,-0.8) -- (10,-0.8)
        node[right] {$R_2$};

    \draw[green!60!black,very thick] (0,-0.8) -- (6,-0.8);
    \draw[green!60!black,very thick] (6,0) -- (10,0);

    \draw[densely dashed] (0,0.8) -- (10,0.8)
        node[right] {$R_0$};

    \draw[densely dashed] (6,0) -- (10,0)
        node[right] {$R_1$};

    \node[left] at (0,0) {Player 2};

\end{scope}

\end{tikzpicture}%
}
\caption{$\mu_0=0,\ \nu_0=2,\ \chi_{20}-\chi_{21}>0$}
\end{minipage}

\end{figure}

\begin{figure}[ht]
\centering

\begin{minipage}{0.48\textwidth}
\centering
\resizebox{\linewidth}{!}{%
\begin{tikzpicture}[>=Stealth,x=1cm,y=1cm,scale=0.8,transform shape]

\begin{scope}
    \draw[->] (0,0) -- (10,0);
    \draw[->] (0,0) -- (0,4) node[above] {$x$};

    \draw[blue,very thick] (0.5,0.3) -- (9.2,3.5);
    \draw[densely dashed] (0,2) -- (5,2);

    \draw[densely dashed] (5,0) -- (5,2);

    \node[left] at (0,2) {$x_{12}$};
\end{scope}

\begin{scope}[yshift=-3cm]

    \draw[densely dashed] (0,0) -- (10,0);

    \draw[densely dashed] (5,-0.8) -- (5,0.8);

    \draw[densely dashed] (0,0.8) -- (10,0.8) node[right] {$R_0$};
    \draw[red,very thick] (0,0) -- (10,0);

    \draw[densely dashed] (0,0) -- (10,0) node[right] {$R_1$};

    \draw[densely dashed] (0,-0.8) -- (10,-0.8) node[right] {$R_2$};

    \node[left] at (0,0) {Player 1};

\end{scope}

\begin{scope}[yshift=-6cm]

    \draw[densely dashed] (0,0) -- (10,0);

    \draw[green!60!black,very thick] (5,-0.8) -- (5,0.8);

    \draw[densely dashed] (0,-0.8) -- (10,-0.8) node[right] {$R_2$};
    \draw[green!60!black,very thick] (0,-0.8) -- (5,-0.8);
    \draw[green!60!black,very thick] (5,0.8) -- (10,0.8);

    \draw[densely dashed] (0,0.8) -- (10,0.8) node[right] {$R_0$};
    \draw[densely dashed] (6,0) -- (10,0) node[right] {$R_1$};

    \node[left] at (0,0) {Player 2};

\end{scope}
\end{tikzpicture}%
}
\caption{$\mu_0=1,\ \nu_0=2,\ \chi_{20}-\chi_{21}<0$}
\end{minipage}
\hfill
\begin{minipage}{0.48\textwidth}
\centering
\resizebox{\linewidth}{!}{%
\begin{tikzpicture}[>=Stealth,x=1cm,y=1cm,scale=0.8,transform shape]

\begin{scope}
    \draw[->] (0,0) -- (10,0);
    \draw[->] (0,0) -- (0,4) node[above] {$x$};

    \node[left] at (0,2.4) {$x_{12}$};

    \draw[blue,very thick] (0.5,0.3) -- (9.2,3.5);
    \draw[densely dashed] (0,2.4) -- (6,2.4);

    \draw[densely dashed] (6,0) -- (6,2.4);

\end{scope}

\begin{scope}[yshift=-3cm]

    \draw[densely dashed] (0,0) -- (10,0);

    \draw[densely dashed] (6,-0.8) -- (6,0.8);

    \draw[densely dashed] (0,0.8) -- (10,0.8) node[right] {$R_0$};
    \draw[red,very thick] (0,0) -- (10,0);

    \draw[densely dashed] (0,0) -- (10,0) node[right] {$R_1$};

    \draw[densely dashed] (0,-0.8) -- (10,-0.8) node[right] {$R_2$};

    \node[left] at (0,0) {Player 1};

\end{scope}

\begin{scope}[yshift=-6cm]

    \draw[densely dashed] (0,0) -- (10,0);

    \draw[green!60!black,very thick] (6,-0.80) -- (6,0);

    \draw[densely dashed] (0,-0.8) -- (10,-0.8) node[right] {$R_2$};
    \draw[green!60!black,very thick] (0,-0.8) -- (6,-0.8);
    \draw[green!60!black,very thick] (6,0) -- (10,0);

    \draw[densely dashed] (0,0.8) -- (10,0.8) node[right] {$R_0$};
    \draw[densely dashed] (6,0) -- (10,0) node[right] {$R_1$};

    \node[left] at (0,0) {Player 2};

\end{scope}

\end{tikzpicture}%
}
\caption{$\mu_0=1,\ \nu_0=2,\ \chi_{20}-\chi_{21}>0$}
\end{minipage}

\end{figure}

\clearpage

\section{Conclusion}

In this paper, we have investigated a two-player zero-sum stochastic differential game with three regimes, focusing on the impact of regime switching on the value function and the optimal strategies of the players. By formulating the problem in terms of the associated Hamilton--Jacobi--Bellman--Isaacs system, we have provided a framework for analyzing the interactions between the different regimes and the resulting switching decisions.

Under suitable assumptions on the model coefficients and switching structure, we have characterized the value function using the viscosity solutions approach and derived explicit representations for a number of relevant configurations. The analysis also provides a qualitative description of the optimal value function in the case where we know the structure of the switching regions without knowing the switching thresholds explicitly. In particular, the results illustrate how the strategic interaction between the two players and the possibility of moving between regimes jointly determine the structure of the optimal policies.

The three-regime setting considered in this work provides a useful framework that captures a richer class of switching mechanisms than the standard two-regime case, while still allowing for an explicit analysis of several configurations. The results obtained here can therefore serve as a basis for studying more general regime-switching stochastic games.

Several directions for future research are possible. One natural extension would be to consider a larger number of regimes and more general switching structures, for which the characterization of the switching regions may become substantially more involved. Another interesting direction would be to develop numerical procedures for computing the switching thresholds when explicit expressions are no longer available. Finally, it would be of interest to investigate extensions involving asymmetric information, non-zero-sum interactions, or regime-dependent diffusion coefficients and switching costs.

\no \textbf{Declarations}

\no \textbf{Conflict of interest.} The authors have not disclosed any competing interests.\\

\end{document}